\documentclass[12pt,a4paper,reqno]{amsart}
\usepackage{amsmath}
\usepackage{amsfonts}
\usepackage{amssymb}
\usepackage{bm} 

\newcommand\eps{\varepsilon}
\newcommand\R{{\mathbf{R}}}

\newcommand\Z{{\mathbf{Z}}}

\theoremstyle{plain}
  \newtheorem{theorem}[subsection]{Theorem}
  \newtheorem{conjecture}[subsection]{Conjecture}
  \newtheorem{proposition}[subsection]{Proposition}
  \newtheorem{lemma}[subsection]{Lemma}

\theoremstyle{remark}
  \newtheorem{remark}[subsection]{Remark}

\theoremstyle{definition}

\include{psfig}

\begin{document}

\title[Quantitative formulation of periodic Navier-Stokes]{A quantitative formulation of the global regularity problem for the periodic Navier-Stokes equation}
\author{Terence Tao}
\address{Department of Mathematics, UCLA, Los Angeles CA 90095-1555}
\email{tao@math.ucla.edu}

\vspace{-0.3in}
\begin{abstract}
The global regularity problem for the periodic Navier-Stokes system
\begin{align*}
\partial_t u + (u \cdot \nabla) u &= \Delta u - \nabla p \\
\nabla \cdot u &= 0 \\
u(0,x) &= u_0(x)
\end{align*}
for $u: \R^+ \times (\R/\Z)^3 \to \R^3$ and $p: \R^+ \times (\R/\Z)^3 \to \R$ asks whether to every smooth divergence-free initial datum $u_0: (\R/\Z)^3 \to \R^3$ there exists a global smooth solution.  In this note we observe (using a simple compactness argument) that this qualitative question is equivalent to the more quantitative assertion that there exists a non-decreasing function $F: \R^+ \to \R^+$ for which one has a local-in-time \emph{a priori} bound
$$ \| u(T) \|_{H^1_x( (\R/\Z)^3 )} \leq F( \|u_0\|_{H^1_x( (\R/\Z)^3 )} )$$
for all $0 < T \leq 1$ and all smooth solutions $u: [0,T] \times (\R/\Z)^3 \to \R^3$ to the Navier-Stokes system.  We also show that this local-in-time bound is equivalent to the corresponding global-in-time bound.
\end{abstract}

\maketitle

\section{Introduction}

Throughout this paper, $\Omega := (\R/\Z)^3$ will denote the standard three-dimensional torus.  This note is concerned with solutions to the periodic Navier-Stokes system\footnote{One can place a viscosity factor $\nu > 0$ in front of the $\Delta u$ term, but this can be easily normalised away by the change of variables $\tilde u(t,x) := \nu u(\nu t,x)$ and $\tilde p(t,x) := \nu^2 p(\nu t, x)$.}
\begin{equation}\label{ns}
\begin{split}
\partial_t u + (u \cdot \nabla) u &= \Delta u - \nabla p \\
\nabla \cdot u &= 0 \\
u(0,x) &= u_0(x)
\end{split}
\end{equation}
where $u: \R^+ \times \Omega \to \R^3$, $p: \R^+ \times \Omega \to \R$ is smooth, and $u_0: \Omega \to \R^3$ is smooth and divergence-free.  As is well known, the pressure $p$ can be eliminated from this system via Leray projections, and so we view this equation as an evolution equation for $u$ alone.

We have the following well-known unsolved conjecture (see e.g. \cite{feff}):

\begin{conjecture}[Global regularity for periodic Navier-Stokes]\label{gpns}  Let $u_0: \Omega \to \R^3$ be smooth and divergence free.  Then there exists a global smooth solution $u: \R^+ \times \Omega \to \R^3$, $p: \R^+ \times \Omega \to \R$ to \eqref{ns}.
\end{conjecture}

There is of course an enormous literature on this and related problems which we will not attempt to survey here; see for instance \cite{bert} for further discussion.

It is well known (e.g. by standard energy methods, see Section \ref{h1-sec} below) that the Navier-Stokes system \eqref{ns} is locally well-posed in the Sobolev space $H^1_x(\Omega)$, with a time of existence depending on the norm $\| u_0 \|_{H^1_x(\Omega)}$ of the initial data, and with the solution smooth on this time interval if the initial data was smooth.  Because of this, a positive answer to Conjecture \ref{gpns} would follow immediately (from standard continuity arguments) from the following long-term \emph{a priori} bound:

\begin{conjecture}[Long-term \emph{a priori} bound for periodic Navier-Stokes]\label{gpns-apriori}  
There exists a non-decreasing function $F: \R^+ \to \R^+$ for which one has an \emph{a priori} bound
\begin{equation}\label{apt-long}
 \| u(T) \|_{H^1_x( \Omega )} \leq F( \|u_0\|_{H^1_x( \Omega )} )
\end{equation}
for all $0 < T < \infty$ and all smooth solutions $u: [0,T] \times \Omega \to \R^3$ and $p: [0,T] \times \Omega \to \R$ to \eqref{ns}.
\end{conjecture}

Alternatively, one could be more modest and only ask for the a priori bound up to time $1$:

\begin{conjecture}[Short-term \emph{a priori} bound for periodic Navier-Stokes]\label{gpns-apriori-short}  
There exists a non-decreasing function $G: \R^+ \to \R^+$ for which one has an \emph{a priori} bound
\begin{equation}\label{apt}
 \| u(T) \|_{H^1_x( \Omega )} \leq G( \|u_0\|_{H^1_x( \Omega )} )
\end{equation}
for all $0 < T \leq 1$ and all smooth solutions $u: [0,T] \times \Omega \to \R^3$ and $p: [0,T] \times \Omega \to \R$ to \eqref{ns}.
\end{conjecture}

The main result of this paper is

\begin{theorem}[Equivalence of qualitative and quantitative regularity conjectures]\label{main} Conjecture \ref{gpns}, Conjecture \ref{gpns-apriori}, and Conjecture \ref{gpns-apriori-short} are all equivalent.
\end{theorem}

This observation is not very deep, being based on well-known compactness properties of the Navier-Stokes flow, and would be unsurprising to the experts\footnote{For other examples of equivalences between qualitative global existence and a priori bounds, see e.g. \cite{bg}, \cite{tao:lens}.}, but the author was not able to find it in the previous literature\footnote{One could make the case, however, that such compactness results are already implicitly present in the literature on compact attractors for nonlinear parabolic equations.}.  The $H^1_x(\Omega)$ norm could be replaced by a variety of other subcritical norms (and probably some critical norms\footnote{For instance, in view of such results as \cite{ess}, \cite{gip}, it would be natural to consider the critical norm $L^3_x$.} also) but we will not pursue such generalisations here.  We stress that Theorem \ref{main} does \emph{not} make any serious progress\footnote{In particular, it does not address at all the fundamental issue of turbulence and supercriticality, since the estimate \eqref{apt} is subcritical and thus out of reach of all known methods.  We remark though that for spherically symmetric classical solutions to the logarithmically supercritical equation $\Box u = u^5 \log(2+u^2)$ in $\R^3$, one has a similarly subcritical \emph{a priori} estimate $\| u(T) \|_{H^2_x(\R^3)} \leq F( \|u(0)\|_{H^2_x(\R^3)} + \|u_t(0) \|_{H^1_x(\R^3)} )$; see \cite{tao-supercrit}.} on Conjecture \ref{gpns} itself,  but it does suggest that this conjecture cannot be solved by purely ``soft'' methods; some quantitative estimates must be involved.  For instance, it emphasises the (already well understood) point that if one seeks to establish Conjecture \ref{gpns} by a regularisation method (approximating $u$ by regularised solutions), it is essential to be able to control those solutions in a subcritical norm such as $H^1_x(\Omega)$ with a bound which is uniform in the choice of regularisation parameter.  Theorem \ref{main} also shows that in order to \emph{disprove} Conjecture \ref{gpns}, it suffices to demonstrate ``norm explosion'' of $H^1_x$, in the sense of \cite{cct}: a sequence of solutions which are bounded in $H^1_x$ at time zero, but are unbounded in $H^1_x$ at later times.

\subsection{Notation}

For any $s \in \R$, we use $H^s_0(\Omega)$ to denote the Banach space of distributions $f$ of mean zero on $\Omega$ whose Fourier transform $\hat f(k) := \int_\Omega e^{-i k \cdot x} f(x)\ dx$ is such that the norm $\|f\|_{H^s_0(\Omega)} := (\sum_{k \in \Z^3 \backslash \{0\}} |k|^{2s} |\hat f(k)|^2)^{1/2}$ is finite.  In practice, these functions will be vector-valued, living in $\R^3$ (or in some cases, $\R^3 \otimes \R^3 \equiv \R^9$).

If $I$ is a compact time interval, $s \in \R$, and $1 \leq p < \infty$, we use $L^p_t H^s_0(I \times \Omega)$ to denote the closure of the smooth mean zero functions on $I \times \Omega$ under the norm
$$ \| u \|_{L^p_t H^s_0(I \times \Omega)} := (\int_I \| u(t) \|_{H^s_0(\Omega)}^p\ dt)^{1/p}.$$
Similarly, we let $C^0_t H^s_0(I \times \Omega) \subset L^\infty_t H^s_0(I \times \Omega)$ be the closure of the smooth mean zero functions on $I \times \Omega$ under the norm
$$ \| u \|_{C^0_t H^s_0(I \times \Omega)} := \| u \|_{L^\infty_t H^s_0(I \times \Omega)}  
:= \sup_{t \in I} \| u(t) \|_{H^s_0(\Omega)}.$$

We use $X \lesssim Y$ to denote the estimate $X \leq CY$ for an absolute constant $C$.  If we need $C$ to depend on a parameter, we shall indicate this by subscripts, thus for instance $X \lesssim_s Y$ denotes the estimate $X \leq C_s Y$ for some $C_s$ depending on $s$.

\subsection{Acknowledgements}

The author is supported by NSF Research Award CCF-0649473 and a grant from the MacArthur Foundation.  The author also thanks Isabelle Gallagher, Nets Katz and Igor Rodnianski for helpful discussions, and Shuanglin Shao for corrections.

\section{Review of $H^1$ theory}\label{h1-sec}

In solving \eqref{ns} we first make a well-known reduction.  For smooth solutions to \eqref{ns}, the mean
$$ \overline{u} = \overline{u}(t) := \int_{\Omega} u(t,x)\ dx$$
is easily seen by integration by parts to be an invariant of the flow (this is just conservation of momentum), thus $\overline{u} = \overline{u_0}$.  By making the change of variables $\tilde u(t,x) := u(t, x - \overline{u_0} t)$ we thus easily reduce to the mean zero case $\overline{u} = \overline{u_0} = 0$.  In that case, as is well-known, we can use Leray projections to eliminate the pressure $p$ from \eqref{ns} and arrive at the equation
$$ \partial_t u = \Delta u + D( u \otimes u )$$
where $D$ is an explicit first order (tensor-valued) Fourier multiplier (or pseudo-differential operator) whose exact form does not need to be made explicit for our arguments, although we will note that the image of $D$ consists entirely of divergence-free mean-zero vector fields.  

By Duhamel's formula we have
\begin{equation}\label{utd}
 u(t) = e^{t\Delta} u_0 + \int_0^t e^{(t-t')\Delta}( D(u \otimes u) )\ dt'.
\end{equation}
Conversely, if $u$ and $u_0$ is smooth and obeys \eqref{utd}, with $u_0$ divergence free, then it is easy to see that $u$ solves \eqref{ns} with an appropriate choice of pressure $p$.

Now suppose $u_0$ is not smooth, but is merely in $H^1_0(\Omega)$.  We define a \emph{strong $H^1_0$ solution} of \eqref{utd} on a time interval $[0,T]$ to be a function $u \in X^1_T$ which obeys \eqref{utd} for all $0 \leq t \leq T$, where for any $s \in \R$, $X^s_T$ is the Banach space $C^0_t H^s_0([0,T] \times \Omega) \cap L^2_t H^{s+1}_0([0,T] \times \Omega)$
with norm
$$ \| u \|_{X_T} := \|u\|_{C^0_t H^s_0([0,T] \times \Omega)} + \|u \|_{L^2_t H^{s+1}_0([0,T] \times \Omega)}.$$
We observe that if $u_0$ is divergence-free then a strong $H^1_0$ solution $u$ of \eqref{utd} must be divergence-free also.

We recall the following standard energy estimate:

\begin{lemma}[Energy estimate]\label{energy}  Let $0 \leq T < \infty$ and $s \in \R$.  If $F \in L^2_t H^{s-1}_0( [0,T] \times \Omega )$ has spatial mean zero and $u_0 \in H^s_0(\Omega)$, then the function $u(t) := e^{t\Delta} u_0 + \int_0^t e^{(t-t')\Delta} F\ dt'$ lies in $X^s_T$ with
$$ \|u\|_{X_T^s} \lesssim_s \|u_0\|_{H^s_0(\Omega)} + \| F \|_{L^2_t H^{s-1}_0([0,T] \times \Omega)}.$$
\end{lemma}

\begin{proof}
The estimate commutes with fractional differentiation operators, so we may take $s=1$. By a limiting argument we may assume that $u_0$ and $F$ (and hence $u$) are smooth, thus $u$ solves the heat equation $u_t + \Delta u = F$ with initial data $u(0) = u_0$.  A standard integration by parts computation then reveals that
$$ \partial_t \| u(t) \|_{H^1_0(\Omega)}^2 \leq -c \|u(t)\|_{H^2_0(\Omega)}^2 + C \int_{\Omega} |\nabla^2 u(t)| |F(t)|\ dt$$
for some absolute constants $c, C > 0$ and all $0 \leq t \leq T$.  By Cauchy-Schwarz we thus have
$$ \partial_t \| u(t) \|_{H^1_0(\Omega)}^2 \leq -c' \|u(t)\|_{H^2_0(\Omega)}^2 + C'\| F(t) \|_{L^2(\Omega)}^2$$
for some other absolute constants $c', C' > 0$.  Integrating this in $t$, we obtain the claim.
\end{proof}

We also recall the smoothing estimate
\begin{equation}\label{smoothing}
\| e^{t\Delta} f \|_{H^{s+\delta}_0(\Omega)} \lesssim_{s,\delta} t^{-\delta/2} \| f \|_{H^s_0(\Omega)}
\end{equation}
for all $0 < t < \infty$, $s \in \R$, and $\delta > 0$, which follows easily from Fourier analysis.

We now review some well known local well-posedness theory for the equation \eqref{utd} in the space\footnote{Much better local well-posedness results are known, of course; see for instance \cite{koch}, \cite{gip}, \cite{auscher}.  Also, the uniqueness aspect of this proposition can be significantly strengthened using the ``weak-strong uniqueness'' theorems in the literature; see \cite{germain} for the most recent results in this direction.} $H^1_0(\Omega)$.

\begin{proposition}[Subcritical parabolic theory]\label{parab}  Let $A > 0$, and set $T := c A^{-4}$ for some small absolute constant $c > 0$.  Then the following statements hold.
\begin{itemize}
\item[(i)] (Existence and uniqueness)  If $u_0 \in H^1_0(\Omega)$ with $\|u_0\|_{H^1_0(\Omega)} \leq A$, then there exists a unique strong solution $u \in X^1_T$ to \eqref{utd}, with the bound
\begin{equation}\label{ctoa}
\|u\|_{X^1_T} \lesssim A.
\end{equation}
Furthermore the solution map $u_0 \mapsto u$ is Lipschitz continuous from the ball $\{ u_0 \in H^1_0(\Omega): \|u_0\|_{H^1_0(\Omega)} \leq A\}$ to $X^1_T$.
\item[(ii)] (Instantaneous regularity) The solution $u$ in (i) is smooth for all positive times $t > 0$.
\item[(iii)] (Compactness)  Let $u$ and $u_0$ be as in (i).  Let $u^{(n)}_0 \in H^1_0(\Omega)$ be a sequence with $\|u^{(n)}_0 \|_{H^1_0(\Omega)} \leq A$ which converges weakly in $H^1_0(\Omega)$ to $u_0$, and let $u^{(n)}$ be the associated strong solutions.  Then for any $0 < \eps < T$, $u^{(n)}$ converges strongly in $C^0_t H^1_0([\eps,T] \times \Omega)$ to $u$.
\end{itemize}
\end{proposition}

\begin{proof}  Fix $A$, and set $T := c A^{-4}$ for some sufficiently small absolute constant $c > 0$.

To prove (i), we fix $u_0$ and let $\Phi: X_T \to X_T$ denote the nonlinear map
$$ \Phi(u)(t) := e^{t\Delta} u_0 + \int_0^t e^{(t-t')\Delta}( D(u \otimes u) )\ dt',$$
thus the task is to establish that $\Phi$ has a unique fixed point which depends in a Lipschitz manner on $u_0$.

To see that $\Phi$ actually maps $X_T$ to $X_T$, we use Lemma \ref{energy} (with $s=1$) followed by the H\"older and Sobolev inequalities (and the fact that $D$ is of order $1$) to compute
\begin{align*}
\| \Phi(u) \|_{X_T} & \lesssim \| u_0 \|_{H^1_0(\Omega)} + \| D(u \otimes u ) \|_{L^2_t L^2_x([0,T] \times \Omega)} \\
&\lesssim A + \| \nabla(u \otimes u) \|_{L^2_t L^2_x([0,T] \times \Omega)} \\
&\lesssim A + T^{1/4} \| |u| |\nabla u| \|_{L^4_t L^2_x([0,T] \times \Omega)} \\
&\lesssim A + c^{1/4} A^{-1} \| u \|_{L^\infty_t L^6_x} \| \nabla u \|_{L^\infty_t L^2_x}^{1/2} \| \nabla u \|_{L^2_t L^6_x}^{1/2} \\
&\lesssim A + c^{1/4} A^{-1} \| u \|_{X_T}^2.
\end{align*}
If $c$ is small enough, we see that $\Phi$ thus maps the ball of radius $CA$ in $X_T$ to itself for some absolute constant $C$, and a similar argument then shows that $\Phi$ is a contraction on that ball if $c$ is small enough, thus yielding the desired unique\footnote{Strictly speaking, this argument only shows uniqueness of the fixed point inside a ball in $X_T$, however one can use continuity arguments (using the fact that a strong solution lies in $C^0_t H^1_0$ and thus varies continuously in time in the $H^1_0$ topology) to extend the uniqueness to all of $X_T$.  Alternatively one can modify the arguments used to prove property (iii) below.} fixed point, which also obeys \eqref{ctoa} and obeys the required Lipschitz property.

To show the regularity property (ii), we observe from the above computations that 
$$ \| D(u \otimes u) \|_{L^4_t L^2_x([0,T] \times \Omega)} \lesssim_A 1.$$
Now from Lemma \ref{energy} we have
$$ \| \int_0^t e^{(t-t')\Delta} F\ dt' \|_{C^0_t H^1_0([0,T] \times \Omega)} \lesssim \| F \|_{L^2_t L^2_x([0,T] \times \Omega)}$$
for any test function $F$, while from \eqref{smoothing} and Minkowski's inequality we easily verify that
$$ \| \int_0^t e^{(t-t')\Delta} F\ dt' \|_{C^0_t H^{2-\sigma}_0([0,T] \times \Omega)} \lesssim_\sigma \| F \|_{L^\infty_t L^2_x([0,T] \times \Omega)}$$
for all $\sigma > 0$, so by complex interpolation we see that
$$ \| \int_0^t e^{(t-t')\Delta} F\ dt' \|_{C^0_t H^{\frac{3}{2}-\sigma}_0([0,T] \times \Omega)} \lesssim_\sigma \| F \|_{L^4_t L^2_x([0,T] \times \Omega)}$$
for all $\sigma > 0$. Applying this with $F := D(u \otimes u)$ and using \eqref{smoothing} we see that
$$ \| u \|_{C^0_t H^{\frac{3}{2}-\sigma}_0([\eps,T] \times \Omega)} \lesssim_{\sigma,\eps,A} 1$$
for all $0 < \eps < T$ and $\sigma > 0$.  Thus we have obtained a smoothing effect of almost half a derivative.  One can continue iterating this argument (using the fractional Leibnitz rule instead of the ordinary one, or alternatively working in H\"older spaces $C^{k,\alpha}$ and using Schauder theory) to get an arbitrary amount of regularity; we omit the standard details.

Now we show the compactness property (iii).  Write $v^{(n)} := u^{(n)} - u$, then we see that $v^{(n)}(0)$ converges weakly to $H^1_0(\Omega)$ and thus strongly in $L^2_x(\Omega)$ (by the Rellich compactness theorem).  Also, by \eqref{utd}, $v^{(n)}$ solves the difference equation
$$ v^{(n)}(t) = e^{t\Delta} v^{(n)}(0) + \int_0^t e^{(t-t')\Delta}( D(v^{(n)} \otimes u) + D(u^{(n)} \otimes v^{(n)} ))\ dt'$$
and so by Lemma \ref{energy} with $s=0$, followed by H\"older, Sobolev, and \eqref{ctoa} (for both $u$ and $u^{(n)}$) we have
\begin{align*}
\| v^{(n)} \|_{X^0_T} &\lesssim \|v^{(n)}(0)\|_{L^2_x(\Omega)} + \| D(v^{(n)} \otimes u) + D(u^{(n)} \otimes v^{(n)} )\|_{L^2_t H^{-1}_0([0,T] \times \Omega)} \\
&\lesssim \|v^{(n)}(0)\|_{L^2_x(\Omega)} + T^{1/4} (\| v^{(n)} \otimes u \|_{L^4_t L^2_x([0,T] \times \Omega)} 
+ \| v^{(n)} \otimes u^{(n)} \|_{L^4_t L^2_x([0,T] \times \Omega)}) \\
&\lesssim
 \|v^{(n)}(0)\|_{L^2_x(\Omega)} + T^{1/4} \| v^{(n)} \|_{L^\infty_t L^2_x([0,T] \times \Omega)}^{1/2}
 \| v^{(n)} \|_{L^2_t L^6_x([0,T] \times \Omega)}^{1/2}  \\
 &\quad\quad
\times ( \| u \|_{L^\infty_t L^6_x([0,T] \times \Omega)} + \| u^{(n)} \|_{L^\infty_t L^6_x([0,T] \times \Omega)} )\\
&\lesssim
 \|v^{(n)}(0)\|_{L^2_x(\Omega)} + c^{1/4} \| v^{(n)} \|_{X^0_T}.
 \end{align*}
Since $v^{(n)}$ has a finite $X^0_T$ norm, we conclude (if $c$ is small enough) that
$$
 \| v^{(n)} \|_{X^0_T} \lesssim \|v^{(n)}(0)\|_{L^2_x(\Omega)}.$$
Now, the above argument in fact shows that
$$ 
\| D(v^{(n)} \otimes u) + D(u^{(n)} \otimes v^{(n)} )\|_{L^4_t H^{-1}_0([0,T] \times \Omega)}
\lesssim_{A,T} \|v^{(n)}(0)\|_{L^2_x(\Omega)}$$
and so by using the interpolation argument as before (but at one lower derivative of regularity) one concludes that
$$ \| v^{(n)} \|_{C^0_t H^{\frac{1}{2}-\sigma}_0([\eps,T] \times \Omega)} \lesssim_{\sigma,\eps,A,T} 
\|v^{(n)}(0)\|_{L^2_x(\Omega)}$$
for all $\sigma > 0$ and $0 < \eps < T$.  Iterating this argument a finite number of times (using the fractional Leibnitz rule) we eventually obtain
$$ \| v^{(n)} \|_{C^0_t H^1_0([\eps,T] \times \Omega)} \lesssim_{\eps,A,T} \|v^{(n)}(0)\|_{L^2_x(\Omega)}$$
for any $0 < \eps < T$.  Since the right-hand side goes to zero as $n$ goes to infinity, the claim (iii) follows.
\end{proof}

\section{Proof of Theorem \ref{main}}

By the discussion of the preceding section we may assume throughout that we are in the mean zero case.

\subsection{Derivation of Conjecture \ref{gpns} from Conjecture \ref{gpns-apriori}.}

This follows immediately from iterating Proposition \ref{parab}.

\subsection{Derivation of Conjecture \ref{gpns-apriori} from Conjecture \ref{gpns-apriori-short}.}

This shall exploit an observation of Leray that any finite energy solution to Navier-Stokes becomes well-behaved after a bounded amount of time, thanks to energy dissipation.

Let $u: [0,T] \times \Omega \to \R^3$ and $p: [0,T] \times \Omega \to \R$ be a smooth solution to \eqref{ns} for some time $0 < T < \infty$.  Write $E := \|u_0\|_{H^1_x(\Omega)}$.  From the well-known energy identity
$$ \partial_t \int_\Omega |u(t,x)|^2\ dx = - 2 \int_\Omega |\nabla u(t,x)|^2\ dx$$
we see that
$$ \| u \|_{L^\infty_t L^2_x([0,T] \times \Omega)} + \| \nabla u \|_{L^2_t L^2_x([0,T] \times \Omega)}
\lesssim_E 1.$$
Let $\eps > 0$ be a small number depending on $E$ to be chosen later.  Suppose for now that $T \geq 1/\eps^2$.  Then by the pigeonhole principle, we can thus find a time $0 \leq T'\leq 1/\eps^2$ such that
$$ \| \nabla u(T') \|_{L^2_x(\Omega)} \lesssim_E \eps.$$
Since $u$ has mean zero, this implies (by Poincar\'e's inequality) that
\begin{equation}\label{hix}
\| u(T') \|_{H^1_x(\Omega)} \lesssim_E \eps.
\end{equation}
Now let $T''$ be any time between $T'$ and $\min(T'+1,T)$.  From Proposition \ref{parab}, we have (for $\eps$ small enough) that
\begin{equation}\label{utt}
 \| u \|_{C^0_t H^1_x([T',T''] \times \Omega)} + \| u \|_{L^2_t H^2_x([T',T''] \times \Omega)} \lesssim_E \| u(T') \|_{H^1_x(\Omega)}.
 \end{equation}
Repeating the computations in that proposition, we conclude in particular that
$$ \| D(u \otimes u) \|_{L^2_t L^2_x([T',T''] \times \Omega)} \lesssim_E \| u(T') \|_{H^1_x(\Omega)} ^2.$$
Using \eqref{utd} and (part of) Lemma \ref{energy}, we conclude that
$$ \| u(T'') - e^{(T''-T')\Delta} u(T') \|_{H^1_x(\Omega)} \lesssim_E \| u(T') \|_{H^1_x(\Omega)} ^2.$$
The least non-trivial eigenvalue of $-\Delta$ on the torus $\Omega$ is 1; since $u(T')$ has mean zero, we thus have
$$ \| e^{(T''-T')\Delta} u(T') \|_{L^2_x(\Omega)} \leq e^{-(T''-T')} \|u(T') \|_{L^2_x(\Omega)}$$
for some $\lambda > 0$.  In particular, if $T'' = T'+1$ then from the triangle inequality we have
$$ \|  u(T'+1) \|_{H^1_x(\Omega)} \leq e^{-1} \|  u(T') \|_{H^1_x(\Omega)} + C_E \| u(T') \|_{H^1_x(\Omega)}^2$$
for some constant $C_E$ depending only on $E$.  From \eqref{hix} we conclude (if $\eps$ is small enough depending on $E$) that
$$ \|  u(T'+1) \|_{H^1_x(\Omega)} \leq \|  u(T') \|_{H^1_x(\Omega)} \lesssim_E \eps.$$
Iterating this inequality as far as we can, and then using \eqref{utt}, we conclude that
\begin{equation}\label{ch}
\| u \|_{C^0_t H^1_x([T',T] \times \Omega)} \lesssim_E \eps.
\end{equation}
On the other hand, since $T' \leq 1/\eps^2$, we can iterate Conjecture \ref{gpns-apriori-short} and obtain
\begin{equation}\label{Ch2}
\| u \|_{C^0_t H^1_x([0,T'] \times \Omega)} \lesssim_{E,\eps} 1.
\end{equation}
Combining \eqref{ch} and \eqref{Ch2} gives Conjecture \ref{gpns-apriori} in the case $T \geq 1/\eps^2$.  In the case $T < 1/\eps^2$, we omit all the steps leading up to \eqref{ch} and simply iterate Conjecture \ref{gpns-apriori-short} as in \eqref{Ch2}.  The claim follows.

\begin{remark} In the mean zero case, the above argument in fact allows us to strengthen \eqref{apt-long} to
\begin{equation}\label{exp-decay}
  \| u(T) - \overline{u_0} \|_{H^1_x( \Omega )} \leq e^{-T} \tilde F( \|u_0 - \overline{u_0} \|_{H^1_x( \Omega )} )
\end{equation}
for some function $\tilde F: \R^+ \to \R^+$ and all solutions $u: [0,T] \times \Omega \to \R^3$, $p: [0,T] \times \Omega \to \R$ to \eqref{ns}.  We leave the details to the reader.
\end{remark}

\subsection{Derivation of Conjecture \ref{gpns-apriori-short} from Conjecture \ref{gpns}}

Suppose for contradiction that Conjecture \ref{gpns-apriori-short} failed.  Carefully unwrapping all the quantifiers and using the axiom of choice\footnote{It is not difficult to eliminate the use of this axiom if desired, for instance by using an explicit (and explicitly well-ordered) countable dense subset of the smooth functions; we omit the details.}, we conclude that there exists $A > 0$ and smooth mean zero solutions $u^{(n)}: [0,T^{(n)}] \times \Omega \to \R^3$, $p^{(n)}: [0,T^{(n)}] \times \Omega \to \R^3$ to \eqref{ns} (and hence \eqref{utd}) with smooth mean zero divergence-free initial data $u^{(n)}(0) = u^{(n)}_0$, where $0 < T^{(n)} < 1$ and $\|u^{(n)}_0\|_{H^1_0(\Omega)} \leq A$ for all $n$, and such that
\begin{equation}\label{untn}
 \lim_{n \to \infty} \| u^{(n)}(T^{(n)}) \|_{H^1_0(\Omega)} = \infty.
 \end{equation}
By passing to a subsequence if necessary we may assume that $\lim_{n \to \infty} T^{(n)} = T$ for some $0 \leq T \leq 1$.  The assertion $T=0$ would contradict \eqref{ctoa}, so we may assume $0 < T \leq 1$.

By passing to a further subsequence we may assume that the $u^{(n)}_0$ are weakly convergent in $H^1_0(\Omega)$ to a limit $u_0 \in H^1_0(\Omega)$, which is also divergence-free.  Applying Proposition \ref{parab}, we can obtain a strong solution $u$ to the Navier-Stokes equation \eqref{utd} with initial datum $u_0$ for some short time $[0,\eps]$ with $\eps > 0$, and that $u$ is smooth, divergence-free and mean zero for all times $0 < t \leq \eps$.  If we then apply Conjecture \ref{gpns} to the initial datum $u(\eps)$, we thus see $u$ can be continued globally as a smooth mean zero solution for all $0 < t < \infty$.  In particular, $u \in C^0_t H^1_0( [0,2T] \times \Omega )$, and so there must exist $A < A' < \infty$ such that \begin{equation}\label{uto}
\sup_{0 \leq t \leq 2T} \|u(t)\|_{H^1_0(\Omega)} \leq A'.
\end{equation}

Now recall that $u^{(n)}_0$ converges weakly in $H^1_0(\Omega)$ to $u_0$ and have the bound $\|u^{(n)}_0 \|_{H^1_0(\Omega)} \leq A \leq A'$.  Then by repeated application of Proposition \ref{parab} (splitting $[0,2T]$ up into intervals of length $T(2A') > 0$, say, and only considering sufficiently large $n$) we see that $u^{(n)}$ converges strongly in $C^0_t H^1_0(I \times \Omega)$ to $u$ for all compact intervals $I \subset (0,2T)$.  But this is inconsistent with \eqref{untn} and \eqref{uto}.  Theorem \ref{main} follows.

\section{Remarks}

It will be clear to the experts that in our main theorem, the torus $\Omega$ could easily be replaced by any other compact surface, and that reasonable boundary conditions can also be added.  It is tempting to extend this result also to the non-periodic case, but the translation invariance causes some failure of compactness which needs to be addressed.  In principle, this can be solved either by concentration compactness, or by requiring that the data be compactly supported or rapidly decreasing; the latter is in fact used in the standard formulations of the non-periodic regularity problem (see e.g. \cite{feff}) but if one were to apply the above arguments to that setting, one encounters the annoying fact that the rapid decrease of the initial data is not preserved by the Navier-Stokes flow (due to the non-local nature of the Leray projection used to eliminate the pressure term).  We will not pursue these matters.

Define the monotone non-decreasing function $F: \R^+ \to [0,+\infty]$ by
$$ F(A) := \sup \{ \| u \|_{C^0_t H^1_x([0,T] \times \Omega)}: \| u_0 \|_{H^1_x(\Omega)} \leq A; 0 \leq T < \infty \}$$
where $u$ ranges over all local smooth solutions $u: [0,T] \times \Omega \to \R^3$ to \eqref{ns} with an initial datum $u_0$ of norm at most $A$.  By Theorem \ref{main}, the periodic Navier-Stokes global regularity conjecture is equivalent to the assertion that $F(A)$ is finite for all $A$.  One can show (by modifying Proposition \ref{parab}) that $F$ is finite for small $A$, and is also right-continuous; thus if the global regularity conjecture fails, there exists a critical value $0 < A_c < \infty$ such that $F(A_c) = +\infty$ and $F(A) < \infty$ for all $A < A_c$.  It is tempting to then use induction-on-energy arguments (as in e.g. \cite{gopher}) to try to analyse solutions at this critical value of the $H^1$ norm, but unfortunately the lack of any conservation law at the $H^1$ level seems to make this idea fruitless.  (In any event, there is nothing special about the $H^1$ norm in this argument; a large number of other subcritical norms would also work here.)

Our arguments yield no information as to the rate of growth of $F$.  The arguments in \cite{kuksin}  (see also \cite{cct}) should at least yield that $F(A)/A$ goes to infinity as $A \to \infty$, but it seems difficult to even obtain $F(A) \gtrsim A^{1+\eps}$ for some $\eps > 0$.  It seems of interest to understand this growth rate better, even at a heuristic level.

It is amusing to note that\footnote{We are indebted to Jean Bourgain for this observation.} for any fixed $0 < A, M < \infty$, the statement $F(A) < M$, if true, can be verified in finite time, by constructing sufficiently accurate and sufficiently numerous numerical solutions, and using some quantitative version of the compactness phenomenon in Proposition \ref{parab} to then rigorously establish $F(A) < M$; we omit the details.  A similar verifiability result holds for the statement $F(A) > M$.  Unfortunately, if $F(A)$ is infinite, there does not appear to be any obvious way to verify this in finite time.

As a final observation, we note from our results that if the Navier-Stokes regularity conjecture is true, then the solution map $u_0 \mapsto u$ is defined as a map from $H^1_0(\Omega)$ to $C^0_t H^1_0([0,\infty) \times \Omega)$; this essentially follows from \eqref{apt-long} and Proposition \ref{parab}.  Furthermore, by working more carefully with the proof of Proposition \ref{parab} and using the exponential decay \eqref{exp-decay} one can show that this map is Lipschitz continuous\footnote{It is important here that the mean $\overline{u_0}$ is fixed.  Small changes in the mean lead to drift in the solution, which is a non-Lipschitz (and non-uniformly continuous) effect.  However, the solution map remains continuous in this case.} on any bounded subset of $H^1_0(\Omega)$; we omit the details.  Thus one can view the Navier-Stokes regularity conjecture as an assertion that the Navier-Stokes flow enjoys global non-perturbative stability in the $H^1$ topology.

\end{document}